\documentclass{amsart}
\usepackage{amsfonts}
\usepackage{amsmath}
\usepackage{tikz}
\usepackage{units}
\usepackage{xcolor}

\newtheorem{theorem}{Theorem}[section]
\newtheorem{lemma}[theorem]{Lemma}
\newtheorem{corollary}[theorem]{Corollary}
\newtheorem{definition}[theorem]{Definition}
\newtheorem{example}{Example}[section]
\newtheorem{observation}[theorem]{Observation}

\title{Semiorders induced by uniform random points}

\author{Csaba Bir\'o and Caroline E.~Boone} 
\address{Department of Mathematics, University of Louisville, Louisville, KY 40220}

\begin{document}

\begin{abstract}
We study semiorders induced by points drawn from a uniform random distribution. Of particular interest in this paper are the probabilities of generating specific semiorders and the equivalence classes they produce. We present a method for calculating the asymptotic probability of inducing these semiorders and describe the qualities which group these semiorders together. We also find a class of semiorders, called \emph{ladders}, whose probabilities can be calculated using the Up/Down numbers.
\end{abstract}

\maketitle

\section{Introduction}

In this paper we explore semiorders induced by points sampled from a uniform random distribution. More specifically, we explore the asymptotic behavior of the probabilities as the underlying interval of the uniform distribution increases in length.

Our work is inspired by that of Kozie\l\ and Sulkowska \cite{uniformrandom} on the uniform generation of general posets as well as that of \L uczak \cite{Luczak} on the first order properties of random posets.  

In the rest of this section we establish some basic definitions related to posets. Many of these concepts and definitions can be found in \cite{trotter}. Note that throughout this paper we will often refer to the poset and its underlying set interchangeably. For example, if $P=(X,<)$ and $x\in X$, then we may also say $x\in P$. It is oftentimes useful to discuss the sets of elements which are less than, greater than, or incomparable to some element in $P$.
\begin{definition}
    Let $x$ be an element of a poset $P$. The \emph{downset} of $x$ with respect to $P$ is defined
    \[D_P(x) = \left\{y\in P\colon\,y<x\right\}\]
    Similarly the \emph{upset} of $x$ with respect to $P$ is defined
    \[U_P(x) = \left\{y\in P\colon\,y>x\right\}\]
    Finally, the \emph{incomparable set} of $x$ with respect to $P$ is defined
    \[I_P(x) = \left\{y\in P\colon\,y\|x\right\}\]
\end{definition}
When the poset is clear in context, we may forgo the subscript and write $D(x)$, $U(x)$, and $I(x)$ to denote the downset, upset, and incomparable set of $x$ in $P$, respectively. 
\begin{definition}
    Let $x$ and $y$ be distinct elements of a poset $P$. If $D_P(x)=D_P(y)$ and $U_P(x)=U_P(y)$ then $x$ and $y$ are \emph{twins}.
\end{definition}

\begin{definition}
    Let $Q=(X_Q,<_Q)$ be a poset and $\{P_x\}_{x\in Q}=\{(X_x,<_x)\}_{x\in Q}$ be a set of posets indexed by the elements of $Q$. Define the set $X = \bigcup_{x\in Q}{X_x}$. Define also the relation $<$ as follows:
    \begin{quote}
        Let $x_1,x_2\in Q$. Let $y_1\in X_{x_1}$ and $y_2\in X_{x_2}$. Then $y_1<y_2$ if 
        \begin{enumerate}
            \item $x_1= x_2$ and $y_1<_{x_1} y_2$ or
            \item $x_1 <_Q x_2$
        \end{enumerate}
    \end{quote}
    \noindent The \emph{lexicographic sum} of $\{P_x\}_{x\in Q}$ over $Q$ is the poset $P = (X,<)$.
\end{definition}
\begin{definition}
    For posets $P$ and $Q$ the \emph{graph sum} of $P$ and $Q$, denoted $P+Q$, is defined as the lexicographic sum of $P$ and $Q$ over the $2$-element antichain.
\end{definition}
\begin{definition}
    A poset $P$ is an \emph{interval order} if there exists a \emph{representation function} $I$ assigning to each $x$ in $P$ a closed interval $I(x)=[\ell_x,r_x]$ on the real line, $\mathbb{R}$, so that $\ell_x>r_y$ if and only if $x<y$ in $P$.
\end{definition}
\begin{theorem}[Fishburn's Theorem]\cite{fishburn} 
    A poset $P$ is an interval order if and only if it does not contain $\mathbf{2}+\mathbf{2}$ as a subposet.
\end{theorem}

\begin{theorem}\cite{kellertrotter}
    Let $P$ be an interval order. Define the set of downsets of $P$ as
    \[D=\left\{D_P(x)\colon\, x\in P\right\}\]
    The set $D$ is totally ordered by by set inclusion. The set of upsets is similarly ordered by inclusion.
\end{theorem}

\begin{definition}
    An interval order $P$ is a \emph{semiorder} if there exists a representation function $I$ such that $I(x) = [\ell_x,\ell_x+1]$ for all $x\in P$.
\end{definition}

\begin{theorem}\cite{scottsuppes}
    A poset $P$ is a semiorder if and only if it contains neither $\mathbf{2}+\mathbf{2}$ nor $\mathbf{3}+\mathbf{1}$ as a subposet.
\end{theorem}

\subsection{Further Definitions}
\begin{definition}
    A semiorder which cannot be written as a lexicographic sum of semiorders over a nontrivial chain is called a \emph{brick}.
\end{definition}
\begin{observation}
    Let $P$ be a semiorder and $I$ a representation function of $P$. Let $H$ be the codomain of $I$.
    \[H = \{I(x)\colon\, x\in P\}\]
    \noindent We can make two observations:
    \begin{enumerate}
        \item If $P$ is a brick, $\bigcup_{I\in H}I$ is an interval on the real line. Specifically, $H = [\min\{\ell_x\colon\, x\in P\},\max\{r_x\colon\, x\in P\}]$.
        \item If $P$ is not a brick, there exists a point $u\in \mathbb{R}$ such that $u\notin I(x)$ for all $x\in P$ and there exist both $x_1$ and $x_2$ in $P$ such that $r_{x_1}<u$ and $\ell_{x_2}>u$.
    \end{enumerate}
\end{observation}

Throughout the paper, we will use the standard notation $f(n)\sim g(n)$ for positive functions $f,g$ to signify that $f(n)$ is asympotic to $g(n)$, i.e. 
    \[\lim_{n\to \infty}\frac{f(n)}{g(n)} = 1.\]

\section{Computing brick probabilities}

Let $L$ be a positive real number. 
Generate a random semiorder the following way. Let $X_1,\ldots,X_n$ be independent uniform random variables on $[0,L]$. Let the random poset $P$ have ground set $\{[X_i,X_i+1]: i=0,\ldots,n\}$, and the ordering is as usual for interval orders: $[a,b]<[c,d]$ if $b<c$. Our goal is to investigate the distribution of $P$.

It is immediately clear that we can't expect this distribution to be anywhere near uniform on the isomorphy classes of semiorders on $n$ points. Indeed, the probability of a chain converges to $1$ as $L\to\infty$, and therefore, the probability of all other semiorders will converge to $0$. The asymptotics of this convergence, however, is interesting and nontrivial.

Our results provide a tool to compute the probability of a given semiorder $P$. First we will show a few technical results to support the probability calculations for bricks. Using the brick probabilities, we are able to expand to lexicographic sums of bricks. Finally we will show that the probabilities of generating some special types of bricks follow a well-known sequence with many other applications. 

\begin{lemma}
Let $P$ be a semiorder. If $x,y$ are not twins, then either $\ell_x<\ell_y$ in every representation, or $\ell_x>\ell_y$ in every representation.
\end{lemma}

\begin{proof}
Let $P$ be a semiorder and $x,y\in P$ such that $x,y$ are not twins.
Since $x$ and $y$ are not twins, either $D(x)\neq D(y)$ or $U(x)\neq U(y)$.
Without loss of generality, assume $D(x)\neq D(y)$. 
Because $P$ is an interval order the downsets of $P$ are linearly ordered by inclusion. 
Thus, either $D(x)\subset D(y)$ or $D(y)\subset D(x)$. 
Without loss of generality, let $D(x)\subset D(y)$.

So, there exists a $z\in D(y)$ such that $z \notin D(x)$. 
Since $z\in D(y)$, it follows that $\ell_z+1< \ell_y$ in every representation of $P$.
Similarly, since $z\notin D(x)$, we know $z\in I(x)\cup U(x)$, so $\ell_z +1 \geq \ell_x$ in every representation of $P$.
Combining these, 
$$\ell_x\leq\ell_z+1<\ell_y$$
Therefore, $\ell_x < \ell_y$ in all representations of $P$.
\end{proof}

\begin{corollary}
Let $P$ be a semiorder. There is a linear extension $(x_1,\ldots,x_n)$ of the vertices, such that
\begin{itemize}
\item In each set of twins, all the elements are consecutive.
\item If $i<j$ and $x_i$, $x_j$ are not twins, then $\ell_{x_i}<\ell_{x_j}$ in every representation.
\end{itemize}
\end{corollary}

Such linear extensions are called \emph{endpoint linear extensions}. (Note that the order of the left and the right endpoints are the same.) These endpoint linear extensions are easy to construct using the following process:
\begin{enumerate}
	\item Choose a representative from each set of twins and delete the other twins from the poset.
	\item Order the elements by increasing size of downset.
	\item If there were any ties in step 2, order the tied elements by decreasing size of upset. 
	\item Insert your previously deleted twins directly to the right (or left) of their respective representative.
\end{enumerate}

\subsection{Lower and upper bound functions}\label{lowerupper}

Let $P$ be an interval order, and $(x_1,\ldots,x_n)$ an endpoint linear extension. We will define two functions, $a$ and $b$, from $P$ to the set of symbolic expressions involving the variables $x_1,\ldots,x_n$. The functions $a$ and $b$ will provide lower and upper bounds, respectively, for the left endpoint of the interval representing $x_i$.

Fix $x_i$, and let $Q$ be the interval order induced by $x_1,\ldots,x_{i-1}$. First we define $a(x_i)$.

First, $a(x_1)=0$ and $b(x_1)=L$, regardless of $P$.

For $i>1$, if $D_Q(x_i)\neq\emptyset$ and $I_Q(x_i)\cap U_Q(x_j)=\emptyset$ where $x_j\in D_Q(x_i)$ such that $j$ is as large as possible, then define $a(x_i)=x_j+1$. Otherwise, define $a(x_i)=x_{i-1}$.

Next, if $I_Q(x_i)=\emptyset$, then define $b(x_i)=L$. If $I_Q(x_i)\neq\emptyset$, then let $x_j\in I_Q(x_i)$ with $j$ as small as possible. Define $b(x_i)=x_j+1$.

\begin{example}
	Consider the semiorder $P$ in Figure~\ref{fig:lowerupper}. We will determine the lower and upper bound functions for each element of $P$.
\end{example}

	\begin{figure}
		\centering
		\begin{tikzpicture}
			\node[circle, fill=black,scale = 0.5,label = below:$x_1$] (A) at (0,0) {};
			\node[circle, fill=black,scale = 0.5,label = below:$x_2$] (B) at (1,0) {};
			\node[circle, fill=black,scale = 0.5,label = left:$x_3$] (C) at (0,1) {};
			\node[circle, fill=black,scale = 0.5,label = right:$x_4$] (D) at (1,1) {};
			\node[circle, fill=black,scale = 0.5,label = above:$x_5$] (E) at (0,2) {};
			\node[circle, fill=black,scale = 0.5,label =above:$x_6$] (F) at (1,2) {};

			\path [-] (A) edge node {} (C);
			\path [-] (A) edge node {} (D);
			\path [-] (B) edge node {} (D);
			\path [-] (D) edge node {} (E);
			\path [-] (D) edge node {} (F);
			\path [-] (C) edge node {} (E);
			\path [-] (C) edge node {} (F);
		\end{tikzpicture}
		\caption{Lower and upper bound functions}\label{fig:lowerupper}
	\end{figure}
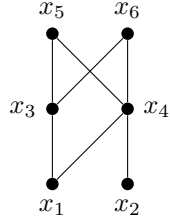

The sequence $(x_1,x_2,x_3,x_4,x_5,x_6)$ as an endpoint linear extension of $P$. Note that the sequence $(x_1,x_2,x_3,x_4,x_6,x_5)$ is also an endpoint linear extension since $x_5$ and $x_6$ are twins. Now we determine the lower bounds for each element:
\begin{itemize}
	\item $x_1$: As stated above, $a(x_1)=0$.
	\item $x_2$: Since $D_Q(x_2) = \emptyset$, we know $a(x_2)=x_1$.
	\item $x_3$: Since $D_Q(x_3) = \{x_1\}$ and $I_Q(x_3)\cap U_Q(x_1)=\emptyset$, we know $a(x_3)=x_1+1$.
	\item $x_4$: Since $D_Q(x_4) = \{x_1,x_2\}$ and $I_Q(x_4)\cap U_Q(x_2)=\emptyset$, we know $a(x_4)=x_2+1$.
	\item $x_5$: Since $D_Q(x_5) = \{x_1,x_2,x_3,x_4\}$ and $I_Q(x_5)\cap U_Q(x_4)=\emptyset$, we know $a(x_5)=x_4+1$.
	\item $x_6$: Since $D_Q(x_6) = \{x_1,x_2,x_3,x_4\}$ and $I_Q(x_6)\cap U_Q(x_4)=\{x_5\}$, we know $a(x_6)=x_5$.
\end{itemize}
Then we determine the upper bounds for each element:
\begin{itemize}
	\item $x_1$: As stated above, $b(x_1)=L$.
	\item $x_2$: Since $I_Q(x_2) = \{x_1\}$, we know $b(x_2)=x_1+1$.
	\item $x_3$: Since $I_Q(x_3) = \{x_2\}$, we know $b(x_3)=x_2+1$.
	\item $x_4$: Since $I_Q(x_4) = \{x_3\}$, we know $b(x_4)=x_3+1$.
	\item $x_5$: Since $I_Q(x_5) = \emptyset$, we know $b(x_5)=L$.
	\item $x_6$: Since $I_Q(x_6) = \{x_5\}$, we know $b(x_6)=x_5+1$.
\end{itemize}

If a set of real numbers chosen from the interval $[0,L]$ satisfies these bouding functions with respect to $P$, then the unit intervals with these left end points generate the semiorder $P$. Note that the converse of this statement is true as well. 

\begin{lemma}\label{thm:boundfuncs}
    Let $P$ be a semiorder and $a$ and $b$ be the bounding functions defined by this semiorder. If the each of the elements $x_i$ in $x_1,\ldots,x_n$ satisfies $a(x_i)<x_i<b(x_i)$, then the intervals $[x_i,x_i+1]$ generate the semiorder $P$.
\end{lemma}
\begin{proof}
    Let $P$ be a semiorder and $L = (x_1,\ldots,x_n)$ an endpoint linear extension of $P$. Assume the left endpoints of an interval representation satisfy the upper and lower bound conditions prescribed by the functions $a$ and $b$. We want to show this interval representation is a representation of $P$. We will do this by induction. Note that if the left endpoint corresponding to $\{x_1\}$ satisfies the upper and lower bound conditions prescribed by the functions $a$ and $b$ then the interval representation is a representation of the subposet of $P$ induced by $\{x_1\}$ since it is just the singleton; in addition, the left endpoint is ordered with respect to $L$. 

Now we fix $x_i$. Assume the left endpoints of $\{x_1,\ldots,x_{i-1}\}$ are ordered with respect to $L$ and form an interval representation of the subposet of $P$ induced by $\{x_1,\ldots,x_{i-1}\}$. Define this subposet to be $Q_i$. We need to show that adding the left endpoint of $x_i$ preserves this order, the interval corresponding to $x_i$ is above every interval corresponding to an element in its downset, and the interval corresponding to $x_i$ is intersecting every interval corresponding to an element in $I_{Q_i}(x_i)$.

First, we show that the addition of $x_i$ preserves the order of $L$ by showing $a(x_i)\geq x_{i-1}$. We have two cases:
\begin{itemize}
	\item $a(x_i)=x_{i-1}$: In which case it is trivially true that $a(x_i)\geq x_{i-1}$
	\item $a(x_i) = x_j+1$ where $x_j \in D_{Q_i}(x_i)$ and $j$ is as large as possible: In this case we also know that $I_{Q_i}(x_i)\cap U_{Q_i}(x_j)=\emptyset$. Therefore, for any $k=j+1,\ldots, i-1$ we know $x_k$ is incomparable to both $x_i$ and $x_j$. Since $Q_i$ is represented by the interval representation formed by the endpoints of $\{x_1,\ldots,x_{i-1}\}$ we know that the intervals corresponding to $x_k$ and $x_j$ intersect for any $k=j+1,\ldots, i-1$. Therefore $x_{i-1}\leq x_j +1$. So we have shown $a(x_i)\geq x_{i-1}$.
\end{itemize}

Next, knowing that the order according to $L$ is preserved, we show that for all $x_j\in D_{Q_i}(x_i)$ we have $x_j+1\leq a(x_i)$. Note that since we already have an interval representation on $\{x_1,\ldots,x_{i-1}\}$, we only need to show this for the greatest possible $j$. Again we have two cases:
\begin{itemize}
	\item $a(x_i)=x_{i-1}$: In this case it must be true that $I_{Q_i}(x_i)\cap U_{Q_i}(x_j)\neq \emptyset$. Therefore, there exists an $x_k$ such that $x_k$ is incomparable to $x_i$ and greater than $x_j$. So, by the endpoint linear extension ordering, 
	$$x_j+1\leq x_k\leq x_{i-1}$$
	\item $a(x_i) = x_j+1$ where $x_j \in D_{Q_i}(x_i)$ and $j$ is as large as possible: This case is trivial.
\end{itemize} 

Finally, we show that for all $x_j\in I_{Q_i}(x_i)$, we have $b(x_i)\leq x_j+1$. If $x_i$ has no incomparables, then this is trivial. Otherwise, by definition we set $b(x_i)=x_k+1$ where $x_k\in I_{Q_i}(x_i)$ and $k$ is as small as possible. Then for any $x_j\in I_{Q_i}(x_i)$, by the endpoint linear ordering we have 
$$b(x_i)=x_k+1\leq x_j+1$$
Therefore we have shown that if $\{x_1,\ldots,x_n\}$ satisfy the bound functions $a$ and $b$, the set of intervals $\{[x_i,x_i+1]\}_{i=1}^n$ induces the semiorder $P$.
\end{proof}

\subsection{The computation theorem}\label{comp}
Using the functions we defined in the previous subsection, we can provide a calculation of the probability that a given semiorder occurs. In order to achieve this, we will first define the following integral.
\[F(P,L)=\frac{1}{L^n}
\int_{a(x_1)}^{b(x_1)}
	\int_{a(x_2)}^{b(x_2)}
		\cdots
		\int_{a(x_n)}^{b(x_n)}
			n!
		\,dx_n
		\ldots
	\,dx_2
\,dx_1\]

This integral will function as an asymptotic approximation for the probability of generating $P$ on the interval $[0,L]$. Note that even for the following small examples, $F(P,L)$ is not equal to the probability of generating $P$ on $[0,L]$. 

Let $P$ be the $2$-element chain and $P'$ be the $2$-element antichain. Then $F(P,L)$ and $F(P',L)$ are defined as follows.
\[F(P,L)=\frac{2}{L^n}
    \int_{0}^{L}
        \int_{x_1+1}^{L}
            1
        \,dx_2
    \,dx_1=\frac{L^2-2L}{L^2}\]
    
\[F(P',L)=\frac{2}{L^n}
    \int_{0}^{L}
        \int_{x_1}^{x_1+1}
            1
        \,dx_2
    \,dx_1=\frac{2L}{L^2}\]
However, we can calculate the true probability of chosing $x_1,x_2$ uniformly from $[0,L]$ using the following figure.
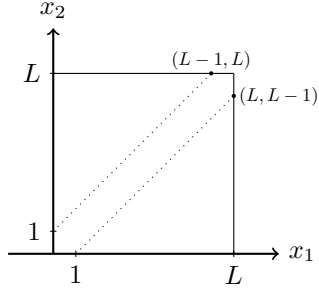
\begin{center}
\begin{tikzpicture}[scale=.6]

  \draw[->,thick] (-1,0) -- (5,0) node[right] {$x_1$};
  \draw[->,thick] (0,0) -- (0,5) node[above] {$x_2$};

  \draw[shift={(.5,0)}, color=black] (0pt,2pt) -- (0pt,-2pt) node[below] {$1$};
  \draw[shift={(0,.5)}, color=black] (2pt,0pt) -- (-2pt,0pt) node[left] {$1$};

  \draw[shift={(4,0)}, color=black] (0pt,2pt) -- (0pt,-2pt) node[below] {$L$};
  \draw[shift={(0,4)}, color=black] (2pt,0pt) -- (-2pt,0pt) node[left] {$L$};

  \draw[-] (4,0) -- (4,4);
  \draw[-] (0,4) -- (4,4);

  \draw[dotted]  (.5,0) -- (4,3.5);
  \draw[dotted]  (0,.5) -- (3.5,4);

  \filldraw [black] (4,3.5) circle (1pt);
  \node[scale=.7] at (5,3.5) {$(L,L-1)$};
  \filldraw [black] (3.5,4) circle (1pt);
  \node[scale=.7] at (3.5,4.3) {$(L-1,L)$};
\end{tikzpicture}

\end{center}

The area inside the square is our sample space for this example. If the sample of $(x_1,x_2)$ lands within or on the dotted lines, they generate an antichain. Otherwise, if $(x_1,x_2)$ is outside the dotted lines, they generate a chain. Since we are choosing uniformly, we can calculate the probability of each of these events by calculating their area and taking it as a proportion of the square's area which is $L^2$. The area outside the dotted lines is easily calculated by doubling the area of one of the triangles. 
\[\Pr(P \text{ generated on } [0,L])=\frac{2\left(\frac{1}{2}(L-1)(L-1)\right)}{L^2}=\frac{L^2-2L+1}{L^2}\]
\[\Pr(P' \text{ generated on } [0,L])=1-\Pr(P \text{ generated on } [0,L])=\frac{2L-1}{L^2}\]
Notice that these are not equal to $F(P,L)$ and $F(P',L)$. This is due to $F(P',L)$ actually calculating the following area outlined in dotted lines.
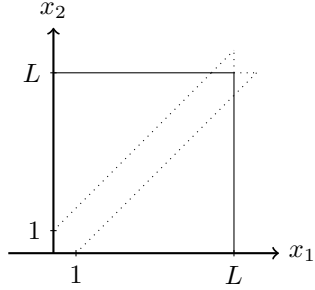
\begin{center}
\begin{tikzpicture}[scale=.6]

  \draw[->,thick] (-1,0) -- (5,0) node[right] {$x_1$};
  \draw[->,thick] (0,0) -- (0,5) node[above] {$x_2$};

  \draw[shift={(.5,0)}, color=black] (0pt,2pt) -- (0pt,-2pt) node[below] {$1$};
  \draw[shift={(0,.5)}, color=black] (2pt,0pt) -- (-2pt,0pt) node[left] {$1$};

  \draw[shift={(4,0)}, color=black] (0pt,2pt) -- (0pt,-2pt) node[below] {$L$};
  \draw[shift={(0,4)}, color=black] (2pt,0pt) -- (-2pt,0pt) node[left] {$L$};

  \draw[-] (4,0) -- (4,4);
  \draw[-] (0,4) -- (4,4);

  \draw[dotted]  (.5,0) -- (4,3.5);
  \draw[dotted] (4,4) -- (4,4.5);
  \draw[dotted]  (0,.5) -- (3.5,4);
  \draw[dotted]  (3.5,4)-- (4,4.5);
  \draw[dotted]  (4,4)-- (4.5,4);
  \draw[dotted]  (4.5,4)-- (4,3.5);
\end{tikzpicture}

\end{center}

This area is equal to that of a parallelogram with base $2$ and height $L$ which gives us the ratio generated by $F(P',L)=\frac{2L}{L^2}$. A similar error occurs in the $F(P,L)$ calculation, in which some area is subtracted instead of added. The goal of these examples is to show that even at the trivial level, $F(P,L)$ is asymptotic to the probability, but not equal. These error terms become more complicated as the semiorder grows, creating a rather technically challenging problem. We also offer an easier way of calculating the probability of generating bricks and then general semiorders later in the paper.

\begin{lemma}\label{thm:bricksarelinear}
	If $P$ is a brick, then $L^n F(P,L)=CL$ where $C \neq 0$. 
\end{lemma}
\begin{proof}
	Let $P$ be a brick of order $n$. The following two statements are true by the definitions of the bound functions:
	\begin{itemize}
	\item $a(x_1)=0$ and $b(x_1)=L$
	\item if $i=2,\ldots,n$ then $x_{i-1}\leq a(x_i) < b(x_i) \leq x_{i-1}+1$
	\end{itemize}

	Since the function of integration in $F(P,L)$ is positive and constant, and the bounds of the $i^{th}$ integral define a subinterval of $[x_{i-1},x_{i-1}+1]$, we know
	\begin{multline*}
		F(P,L)\leq \frac{1}{L^n}
	\int_{0}^{L}
		\int_{x_1}^{x_1+1}
			\cdots
			\int_{x_{n-1}}^{x_{n-1}+1}
				n!
			\,dx_n
			\ldots
		\,dx_2
	\,dx_1\\=\frac{n!}{L^n}
	\int_{0}^{L}
		\int_{x_1}^{x_1+1}
			\cdots
			\int_{x_{n-2}}^{x_{n-2}+1}
				(x_{n-1}+1-x_{n-1})
			\,dx_{n-1}
			\ldots
		\,dx_2
	\,dx_1\\=\frac{n!}{L^n}
	\int_{0}^{L}
		1
	\,dx_1=\frac{n!L}{L^n}
	\end{multline*}
	
	Therefore, $L^n F(P,L)$ is at most linear in terms of $L$. It is left to show that $L^n F(P,L)$ does not have a nonzero constant term.

	We know that $\int_{a(x_2)}^{b(x_2)}
	\cdots
	\int_{a(x_n)}^{b(x_n)}
		n!
	\,dx_n
	\ldots
\,dx_2$ is a polynomial function of $x_1$ due to the constant integrand and the bounding functions. So, when we integrate this function over $[0,L]$, we cannot get a nonzero constant term. 

Finally, note that $L^n F(P,L)\neq 0$ because the integrand is positive and the bounds of each integral are in increasing order. 

So, we have shown $L^n F(P,L)=CL$ where $C\neq 0$.

\end{proof}
The following definition of \textit{order statistics} and the formula for their joint density function will be necessary in the proof of Theorem~\ref{thm:brickprob}.
\begin{definition}
	Let $X_1,\ldots,X_n$ be $n$ independent and identically distributed continuous random variables having a common density $f$ and distribution function $F$. Define $X_{(i)}$ with $i=1,\ldots,n$ to be the $i$ th smallest of $X_1,\ldots,X_n$. The ordered values $X_{(1)}\leq\ldots\leq X_{(n)}$ are known as the \textup{order statistics} corresponding to the random variables $X_1,\ldots,X_n$.
\end{definition}

The following formula for the joint density function of the order statistics is well-documented in the literature and will also be necessary for the proof of Theorem~\ref{thm:brickprob}. A proof is available in \cite{sheldonross}. 

\begin{theorem}\label{thm:orderstats}
	Let $X_1,\ldots,X_n$ be $n$ independent and identically distributed continuous random variables having a common density $f$. The joint density function of the order statistics $X_{(1)},\ldots,X_{(n)}$ is 
	$$f_{X_{(1)},\ldots,X_{(n)}}(x_1,\ldots,x_n)=n!\,f(x_1)\cdots f(x_n)$$
\end{theorem}
Using order statistics, we will begin by showing the probability calculation for bricks. 
\begin{theorem}\label{thm:brickprob}
    If $P$ is a brick then $\Pr(P \text{ generated on } [0,L])=F(P,L)+\frac{C}{L^n}$ where $C\in \mathbb{R}$.
\end{theorem}
\begin{proof}
	Let $P$ be a brick. We begin by partitioning the probability into two mutually exculsive events.
    \begin{multline}
        \Pr(P \text{ generated on } [0,L])
        = \Pr(P \text{ generated on } [0,L] \text{ and } x_{(1)}> L-n)\\+\Pr(P \text{ generated on } [0,L] \text{ and } x_{(1)}\leq L-n)
    \end{multline}
    Calculating each of these summands:
    \begin{multline}
        \Pr(P \text{ generated on } [0,L] \text{ and } x_{(1)}> L-n) 
        \\= \Pr(P \text{ generated on } [0,L]\ |\  x_{(1)}> L-n)\Pr(x_{(1)}> L-n)
        \\=\Pr(P \text{ generated on } [L-n,L])\Pr(x_{(1)}> L-n)
        \\=\Pr(P \text{ generated on } [0,n])\Pr(x_{(1)}> L-n)
    \end{multline}
	Since $\Pr(P \text{ generated on } [0,n])$ does not depend on $L$, it must be the case that this probability is constant with respect to $L$. We will set $\Pr(P \text{ generated on } [0,n]) = C'$. Then, $\Pr(x_{(1)}> L-n)$ is simply the probability that each of $x_1,\ldots,x_n$ are in the interval $[L-n,L]$ which is $(\frac{n}{L})^n$. Thus,
	\[\Pr(P \text{ generated on } [0,L] \text{ and } x_{(1)}> L-n)=C'\left(\frac{n}{L}\right)^n=\frac{C}{L^n}\]
    Then, to calculate the second summand, we restrict $x_{(1)}$ to the interval $[0,L-n]$ and for the remaining integrals' bounds, we use the bounds $a$ and $b$ defined above. By Lemma~\ref{thm:boundfuncs}, these bounds ensure that the interval corresponding to each element will be correctly oriented with respect to the other intervals. We also know that if these bounds are satisfied and $P$ is a brick, $x_n-x_1<n$ so in this scenario, $0\leq a(x_i)<b(x_i)\leq L$ for all $i=1,\ldots,n$. In addition, using Theorem~\ref{thm:orderstats} we know the following integral will correctly calculate the given probability.
    \begin{multline}
        \Pr(P \text{ generated on } [0,L] \text{ and } x_{(1)}\leq L-n)
        \\=
        \int_{0}^{L-n}\int_{a(x_2)}^{b(x_2)}\cdots\int_{a(x_n)}^{b(x_n)} \frac{n!}{L^n} d\,x_1\ldots d\,x_n
        =\frac{(L-n)^n}{L^n}F(P,L-n)
    \end{multline}
	We also know by Lemma~\ref{thm:bricksarelinear} that $F(P,L-n)=\frac{C_P (L-n)}{(L-n)^n}$. Therefore,
	\[\Pr(P \text{ generated on } [0,L] \text{ and } x_{(1)}\leq L-n)=\frac{C_P (L-n)}{L^n}\]
    Plugging both of these back into (1),
    \begin{multline}
        \Pr(P \text{ generated on } [0,L])
        = \frac{C}{L^n}+\frac{C_P (L-n)}{L^n} \\= \frac{C_P L}{L^n}+\frac{C-C_P n}{L^n} = \frac{F(P,L)}{L^n}+\frac{C-C_P n}{L^n}
    \end{multline}
    So, the error between $F(L,P)$ and $\Pr(P \text{ generated on } [0,L])$ is $\frac{C-C_P n}{L^n}$ which is a constant with respect to $L$. 
\end{proof}

This theorem clearly implies the following corollary.
\begin{corollary}\label{brickcorr}
    If $P$ is a brick on $n$ elements then $L^n\Pr(P)\sim F(P,L)$.
\end{corollary}

Now that we can calculate the probability of generating a brick on $[0,L]$, these brick probabilities can be used to calculate the probabilities of nonbricks. 
\section{Lexicographic sums over a chain}
Let $P$ be an semiorder, which is a lexicographic sum of semiorders $P_1$, $P_2$, \dots, $P_k$ over a chain of length $k$. This section describes how find the asymptotic probability of generating $P$ based on the probabilities of the subsemiorders $P_i$. Recall that these probabilities are always rational functions of $L$ with $L^n$ in the denominator.

\begin{theorem}\label{thm:coefficients} Let $P$ be a semiorder of order $n$, which is a lexicographic sum of bricks $P_1$, $P_2$, \dots, $P_k$ over a chain. Let $c_i$ and $e_i$ defined by
\[
L^{|P_i|}\Pr(P_i)= c_i L +e_i.
\]
Then the probability of generating $P$ on the interval $[0,L]$ is
\[
L^n \Pr(P \text{ generated on } [0,L])=\binom{n}{|P_1|,|P_2|,\ldots,|P_k|}\frac{c_1\cdots c_k }{k!}L^k + E(L)
\]
Where $E(L)$ is a polynomial function of $L$ with $\deg(E)<k$.
\end{theorem}

\begin{proof}
    We will prove this theorem by induction. We will begin with the base case of $k=1$, or equivalently, the case in which $P$ is a brick. By Theorem~\ref{thm:brickprob} we know that the following is true with $\deg(E)=0$.
    \[L^n \Pr(P \text{ generated on } [0,L]) = c_1 L +e_1 = \binom{n}{n}\frac{c_1}{1!}L^1+E(L)\]
    Suppose the theorem is true for all semiorders which are lexicographic sums of $k-1$ bricks over the $(k-1)$-chain. Let $P$ be a semiorder which is a lexicographic sum of $P_1,\ldots,P_k$ over the $k$-chain. Define $m=|P_1|$. We write the following integral using the Law of Total Probability.
    \begin{multline}
        \Pr(P \text{ generated on } [0,L]) \\= \int_0^L \Pr(P \text{ generated on } [0,L]\ |\ x_{(m)}=x)f_{x_{(m)}}(x)\,dx
    \end{multline}
    This probability is the instersection between two independent probabilities.
    \begin{multline}\label{thm:partitionedint}
        \int_0^L \Pr(x_{(1)},\ldots,x_{(m)} \text{ form } P_1 \text{ on } [0,x]\ |\ x_{(m)}=x)\\\Pr(x_{(m+1)},\ldots,x_{(n)}\text{ form } P\setminus P_1 \text{ on } [x+1,L]\ |\ x_{(m)}=x)\\f_{x_{(m)}}(x)\,dx
    \end{multline}
    We will then partition this integral over three intervals: $[0,m],\ [m,L-n+m],\ [L-n+m,L]$. For the first interval, define its integral to be $I_1$.
    \begin{multline}\label{thm:firstint}
        I_1 = \int_0^m \Pr(x_{(1)},\ldots,x_{(m)} \text{ form } P_1 \text{ on } [0,x]\ |\ x_{(m)}=x)\\\Pr(x_{(m+1)},\ldots,x_{(n)}\text{ form } P\setminus P_1 \text{ on } [x+1,L]\ |\ x_{(m)}=x)\\f_{x_{(m)}}(x)\,dx
    \end{multline}
    This interval is constant length with respect to $L$. So, since $P_1$ is a brick, so must $P\setminus x_m$ be. Therefore $L^n \Pr(P_1 \setminus x_m\text{ generated on } [0,L]) = C'L+D'$.
    \begin{multline}
        \Pr(x_{(1)},\ldots,x_{(m)} \text{ form } P_1 \text{ on } [0,x]\ |\ x_{(m)}=x)\\\leq \Pr(x_{(1)},\ldots,x_{(m-1)} \text{ form } P_1\setminus x_m \text{ on } [0,x]\ |\ x_{(m)}=x)\\-\Pr(x_{(1)},\ldots,x_{(m-1)} \text{ form } P_1\setminus x_m \text{ on } [0,x-1]\ |\ x_{(m)}=x)
        \\ \leq \frac{C'x+D'}{x^{m-1}}-\frac{C'(x-1)+D'}{x^{m-1}}
        = \frac{C'}{x^{m-1}}
    \end{multline}

    Plug this back into (\ref{thm:firstint}).
    \begin{multline}
        \int_0^m \Pr(x_{(1)},\ldots,x_{(m)} \text{ form } P_1 \text{ on } [0,x]\ |\ x_{(m)}=x)\\\Pr(x_{(m+1)},\ldots,x_{(n)}\text{ form } P\setminus P_1 \text{ on } [x+1,L]\ |\ x_{(m)}=x)\\f_{x_{(m)}}(x)\,dx
        \\\leq \int_0^m \frac{C'}{x^{m-1}}\Pr(x_{(m+1)},\ldots,x_{(n)}\text{ form } P\setminus P_1 \text{ on } [x+1,L]\ |\ x_{(m)}=x)\\f_{x_{(m)}}(x)\,dx
        \\=\int_0^m \frac{C'}{x^{m-1}}\left[\binom{n-m}{|P_2|,\ldots,|P_k|}\frac{c_2\cdots c_k }{(k-1)!}(L-x-1)^{k-1} + E(L-x-1)\right]\\\left(\frac{1}{L-x}\right)^{n-m}\left(\frac{n!}{(n-m)!(m-1)!}\right)\left(\frac{x}{L}\right)^{m-1}\left(\frac{1}{L}\right)\left(\frac{L-x}{L}\right)^{n-m}\,dx
    \end{multline}
    \begin{multline*}
        =\frac{C'}{L^n}\left(\frac{n!}{(n-m)!(m-1)!}\right)
        \\\int_0^m \left[\binom{n-m}{|P_2|,\ldots,|P_k|}\frac{c_2\cdots c_k }{(k-1)!}(L-x-1)^{k-1} + E(L-x-1)\right]\,dx
    \end{multline*}
    Split this into two integrals.
    \begin{multline*}
        =\frac{C'}{L^n}\left(\frac{n!}{(n-m)!(m-1)!}\right)
        \binom{n-m}{|P_2|,\ldots,|P_k|}\frac{c_2\cdots c_k }{(k-1)!}\int_0^m (L-x-1)^{k-1}\,dx
        \\+\frac{C'}{L^n}\left(\frac{n!}{(n-m)!(m-1)!}\right)
        \int_0^m E(L-x-1)\,dx
    \end{multline*}
    Then perform a substitution with $u=L-x-1$.
    \begin{multline*}
        =\frac{C'}{L^n}\left(\frac{n!}{(n-m)!(m-1)!}\right)
        \binom{n-m}{|P_2|,\ldots,|P_k|}\frac{c_2\cdots c_k }{(k-1)!}\int_{L-m-1}^{L-1} u^{k-1}\,du
        \\+\frac{C'}{L^n}\left(\frac{n!}{(n-m)!(m-1)!}\right)
        \int_{L-m-1}^{L-1} E(u)\,du
        \\=\frac{C'}{L^n}\left(\frac{n!}{(n-m)!(m-1)!}\right)
        \binom{n-m}{|P_2|,\ldots,|P_k|}\frac{c_2\cdots c_k }{(k-1)!}\frac{1}{k}\left[(L-1)^k-(L-m-1)^k\right] 
        \\+\frac{C'}{L^n}\left(\frac{n!}{(n-m)!(m-1)!}\right)(\mathbf{E}_1(L-1)-\mathbf{E}_1(L-m-1))
    \end{multline*}
    Since $\deg(E)<k-1$ we know that the degree of its antiderivative, $\deg(\mathbf{E})<k$. Also notice that in the first summand of this integral, the $L^k$ terms also cancel out. Therefore, the integral defined in (\ref{thm:firstint}) is a polynomial function of $L$, and $\deg(I_1)<k$.

    We define he second partition of the integral in (\ref{thm:partitionedint}) as $I_2$.
    \begin{multline}\label{thm:secondint}
        I_2 = \int_m^{L-n+m} \Pr(x_{(1)},\ldots,x_{(m)} \text{ form } P_1 \text{ on } [0,x]\ |\ x_{(m)}=x)\\\Pr(x_{(m+1)},\ldots,x_{(n)}\text{ form } P\setminus P_1 \text{ on } [x+1,L]\ |\ x_{(m)}=x)\\f_{x_{(m)}}(x)\,dx
        \\=\int_m^{L-n+m} \frac{c_1/m}{x^{m-1}}\left[\binom{n-m}{|P_2|,\ldots,|P_k|}\frac{c_2\cdots c_k }{(k-1)!}(L-x-1)^{k-1} + E(L-x-1)\right]\\\left(\frac{1}{L-x}\right)^{n-m}\left(\frac{n!}{(n-m)!(m-1)!}\right)\left(\frac{x}{L}\right)^{m-1}\left(\frac{1}{L}\right)\left(\frac{L-x}{L}\right)^{n-m}\,dx
        \\=\binom{n}{|P_1|,|P_2|,\ldots,|P_k|}\frac{c_1\cdots c_k }{(k-1)!}\left(\frac{1}{L^n}\right)\int_m^{L-n+m} (L-x-1)^{k-1} \,dx + \mathbf{E}_2(L)
        \\=\binom{n}{|P_1|,|P_2|,\ldots,|P_k|}\frac{c_1\cdots c_k }{(k-1)!}\left(\frac{1}{L^n}\right)\frac{1}{k}\left((L-m-1)^k-(n-m-1)^k\right)+\mathbf{E}_2(L)
        \\=\binom{n}{|P_1|,|P_2|,\ldots,|P_k|}\frac{c_1\cdots c_k }{k!}\left(\frac{L^k}{L^n}\right)+\mathbf{E}_2(L)
    \end{multline}
    Using a similar argument to the last partition of this integral, the portion containing $E(L-x-1)$ only contributes to the $\mathbf{E}_2(L)$ portion which has degree less than $k$.
    
    Finally, we consider the third partition of (\ref{thm:partitionedint}) where $\mathbf{E}_3(L)$ is a polynomial function of degree less than $k$.
    \begin{multline}\label{thm:thirdint}
        I_3 = \int_{L-n+m}^L \Pr(x_{(1)},\ldots,x_{(m)} \text{ form } P_1 \text{ on } [0,x]\ |\ x_{(m)}=x)\\\Pr(x_{(m+1)},\ldots,x_{(n)}\text{ form } P\setminus P_1 \text{ on } [x+1,L]\ |\ x_{(m)}=x)\\f_{x_{(m)}}(x)\,dx
        \\ = \int_{L-n+m}^L \frac{c_1/m}{x^{m-1}}\Pr(x_{(m+1)},\ldots,x_{(n)}\text{ form } P\setminus P_1 \text{ on } [x+1,L]\ |\ x_{(m)}=x)\\\left(\frac{n!}{(n-m)!(m-1)!}\right)\left(\frac{x}{L}\right)^{m-1}\left(\frac{1}{L}\right)\left(\frac{L-x}{L}\right)^{n-m}\,dx
        \\=\left(\frac{n!}{(n-m)!(m-1)!}\right)\frac{c_1/m}{L^n}\int_{L-n+m}^L \Pr(x_{(m+1)},\ldots,x_{(n)}\text{ form } P\setminus P_1 \text{ on }\\ [x+1,L]\ |\ x_{(m)}=x)(L-x)^{n-m}\,dx
        \\=\mathbf{E}_3(L)
    \end{multline}
    
    Add $I_1,\ I_2,$ and $I_3$ back together to reconstruct (\ref{thm:partitionedint}).
    \begin{multline}
        \Pr(P \text{ generated on } [0,L]) \\= \binom{n}{|P_1|,|P_2|,\ldots,|P_k|}\frac{c_1\cdots c_k }{k!}\left(\frac{L^k}{L^n}\right)+\mathbf{E}(L)
    \end{multline}
    Here we have that $\mathbf{E}(L)=\mathbf{E}_1(L)+\mathbf{E}_2(L)+\mathbf{E}_3(L)$ and $\deg(\mathbf{E})<k$. So we have shown the theorem for all $k$.
\end{proof}

This theorem also implies the following few corollaries having to do with asymptotics. 
\begin{corollary}
Let $P$ be a semiorder of order $n$, which is a lexicographic sum of bricks $P_1$, $P_2$, \dots, $P_k$ over a chain. Let $c_i$ and $e_i$ defined by
\[
L^{|P_i|}\Pr(P_i)\sim c_i L^{e_i}.
\]
Then 
$$L^n\Pr(P)\sim \binom{n}{|P_1|,|P_2|,\ldots,|P_k|}\frac{c_1\cdots c_k }{k!}L^k$$.
\end{corollary}
Similarly, we can determine the exponent in the numerator of the probability without knowing exactly which bricks are lexicographically summed to create $P$. 
\begin{corollary}
Let $P$ be a semiorder of order $n$, which is a lexicographic sum of semiorders $P_1$, $P_2$, \dots, $P_k$ over a chain. Let $c_i$ and $e_i$ defined by
\[
L^{|P_i|}\Pr(P_i)\sim c_i L^{e_i}.
\]
Then 
$$L^n\Pr(P)= \Theta\left(L^{e_1+\cdots+e_k}\right).$$
\end{corollary}

Further, since the asymptotic probability of a lexicographic sum of bricks over a chain is determined by the asymptotic probabilities of those bricks, we can state the following corollary as well.

\begin{corollary}
    Let $P$ and $P'$ be semiorders of order $n$, which are lexicographic sums of bricks $P_1$, $P_2$, \dots, $P_k$ and $P'_1$, $P'_2$, \dots, $P'_k$, respectively, over a chain. Let $c_i, c'_i$ and $e_i, e'_i$ be defined by
    \[
    L^{|P_i|}\Pr(P_i)\sim c_i L^{e_i} \text{ and } L^{|P'_i|}\Pr(P'_i)\sim c'_i L^{e'_i}.
    \]
    If there exists a bijective map $\alpha: \{1,\ldots,k\}\to \{1,\ldots,k\}$ such that $\Pr(P_i)\sim \Pr(P'_{\alpha(i)})$, then $\Pr(P)\sim \Pr(P')$.

\end{corollary}

Finally, we show that the integral we defined in Subsection~\ref{comp} is asmyptotic to the probability for nonbricks. 
\begin{corollary}
    If $P$ is a semiorder, $\Pr(P)\sim F(P,L)$.
\end{corollary}
\begin{proof}
    We showed this corollary for bricks in Corollary~\ref{brickcorr}. We can now extend this to non-bricks. Let $P$ be a lexicographic sum of bricks $P_1$, $P_2$, \dots, $P_k$ over a chain. From Theorem~\ref{thm:coefficients}, we know that $L^n\Pr(P)\sim CL^k$ for some $C\in\mathbb{R}$. Construct $F(P,L)$ using the lower and upper bound functions. Note that the first element $x_{i_1}\in P_i$ in any endpoint linear extension will have $b(x_{i_1})=L$. 

    For the purposes of this proof, reassign $b(x_{i_1})=L-(n-i_1)$. Then define $F'(P,L)$ to use these new upper bounds. This calculates the following probability exactly:
    \[F'(P,L)=\Pr(P \text{ and }  x_1<L-n+1 \text{ and } \ldots \text{ and } x_{k_1}<L-(n-k_1))\]
    Then we can establish the following equation.
    \[\Pr(P)=F'(P,L)+\Pr(P \text{ and } x_{i_1}>L-(n-i_1) \text{ for at least one }i=1,\ldots, k)\]
    Note that if any $x_{i_1}>L-(n-i_1)$, then $x_{k_1}>L-(n-i_1)>L-n$. So we know an upper bound this probability. 
    \begin{multline*}
    \Pr(P \text{ and } x_{i_1}>L-(n-i_1) \text{ for at least one }i=1,\ldots, k)
    \\\leq \Pr(P \text{ and } x_{k_1}>L-n)
    \\\leq \Pr(P\setminus P_k)\Pr(P_k \text{ on } [L-n,L])
    \\\leq \Pr(P\setminus P_k)\Pr(P_k \text{ on } [0,n]) = \Theta(L^{k-1-n})
    \end{multline*}
    So it is clear that every term of $L^{k-n}$ in $\Pr(P)$ must be supplied from $F'(P,L)$. Since $F'(P,L)$ necessarily shares its  leading coefficient with $F(P,L)$ due to the nature of the integral, we have shown that $\Pr(P)\sim F(P,L)$.
\end{proof}

\section{Ladders}
In this section, we will focus on a special type of bricks called \emph{ladders}. Ladders are important because, as subposets, they restrict the maximum possible range for a representation of a given brick.
\begin{definition}
    A semiorder $P$ with endpoint linear extension $L=(x_1,\ldots,x_n)$ is called the \emph{ladder} on $n$ elements if for all $x_i\in P$, $D(x_i)=\{x_j\in P:\, j<i-1\}$ and $U(x_i)=\{x_j\in P:\, j>i+1\}$.
\end{definition}
The probability coefficients of generating ladders follows a famous number sequence called the Euler numbers, Up/Down numbers, or Zigzag numbers. We give a robust overview of this sequence in the following subsection. 
\subsection{Up/Down Numbers}
This sequence of numbers count a variety of scenarios. The $n$-th term of this sequence, $A_n$ represents each of the following.
\begin{itemize}
    \item half the number of alternating permutations on $n$ letters \cite{andre}
    \item the number of linear extensions of the zig-zag poset on $n$ elements \cite{stanley}
    \item the number of increasing $0-1-2$ trees on $n$ vertices \cite{callan}
\end{itemize}

Many more uses of this sequence can be found on OEIS page A000111. These numbers have both a recursive definition and an exponential generating function.

\begin{definition}
    The Up/Down number sequence is defined by the following recursion.
    \[2A_{n+1}=\sum_{k=0}^{n}\binom{n}{k}A_kA_{n-k}\] 
    The sequence also has exponential generating function $\sec(x)+\tan(x)$. 
\end{definition}

Additionally, the ratio of $nA_{n-1}$ to $A_n$ as $n$ goes to infinity has a closed form solution \cite{limit}.
\[\lim_{n\to\infty}\frac{nA_{n-1}}{A_n}=\frac{\pi}{2}\]

For our purposes, the most important scenario that this sequence can be used to represent is outlined in the following theorem stated on OEIS. We provide a proof for completeness. 

\begin{theorem}\label{thm:eulergivesalt}
    The probability that $n$ numbers $\{v_1,\ldots,v_n\}$ chosen uniformly from $[0,1]$ satisfy $v_1>v_2<v_3>\ldots v_n$ is $\frac{A_n}{n!}$.
\end{theorem}

\begin{proof}
    We choose $n$ numbers $\{v_1,\ldots,v_n\}$ uniformly from $[0,1]$. Let $\{v_{(1)},\ldots,v_{(n)}\}$ be the order statistics of this sampling.

    Define the map $\sigma: \{v_1,\ldots,v_n\}\to \{1,\ldots,n\}$ by $\sigma(v_i)=j$ if $v_i=v_{(j)}$. The permutation $(\sigma(v_1),\sigma(v_2),\ldots,\sigma(v_n))$ is an alternating permutation of $\{1,2,\ldots,n\}$ begining with a decrease if and only if $\{v_1,\ldots,v_n\}$ satisfy $v_1>v_2<v_3>\ldots v_n$. Therefore there are $A_n$ permutations which satisfy this condition. Thus,
    \[\Pr(v_1>v_2<v_3>\ldots v_n)= \frac{A_n}{n!}\]
\end{proof}
\subsection{Ladder Probabilities}
Using the previous definitions of the Up/Down numbers, we can prove the following theorem pertaining to the probability of generating a ladder on $n$ elements.

\begin{theorem}\label{thm:eulerladder}
    If $P$ is a ladder of order $n$, then $L^n\Pr(P)\sim n A_{n-1} L$ where $A_{n-1}$ is the $(n-1)^{th}$  Up/Down Number.
\end{theorem}

\begin{proof}
    By Theorem~\ref{thm:brickprob} and since $P$ is a ladder, we can construct the following integral.
    \begin{multline}\label{thm:ladderprob}
        \frac{L^n}{n!} \Pr(P)=
        \int_{0}^{L}
            \int_{x_1}^{x_1+1}
                    \int_{x_1+1}^{x_2+1}
                        \int_{x_2+1}^{x_3+1}
                            \cdots
                    \int_{x_{n-2}+1}^{x_{n-1}+1}
                    \,dx_n
                    \ldots
                    \,dx_4
                \,dx_3
            \,dx_2
        \,dx_1+e
    \end{multline}
    By Theorem~\ref{thm:eulergivesalt}, the probability that $n$ numbers $\{v_1,\ldots,v_n\}$ chosen uniformly from $[0,1]$ satisfy $v_1>v_2<v_3>\ldots v_n$ is $A_n/n!$. Writing the integral which calculates this probability gives us the following. 
    \begin{equation}
    \label{thm:altint}
    {\frac{A_{n-1}}{(n-1)!}=
    \int_{0}^{1}
        \int_{0}^{v_1}
            \int_{v_2}^{1}
                \cdots
            \,dv_3
        \,dv_2
    \,dv_1}
    \end{equation}
    Notice that both (\ref{thm:ladderprob}) and (\ref{thm:altint}) calculate the volumes of polyhedra. We will transform the variables $x_2,\ldots, x_n$ using a scaling matrix $\boldsymbol{A}$ with $|\det(\boldsymbol{A})|$ and a shifting vector. Then we can show that these integrals calculate the same volume. We will first define $\boldsymbol{A}$. 
    \[
    \boldsymbol{A}=\left[ {\begin{array}{ccccccc}
        1 & 0 & 0 & 0 &\cdots & \cdots & 0 \\
        1 & -1 & 0 & 0 &\cdots & \cdots & 0\\
        0 & -1 & 1 & 0 &\cdots & \cdots & 0\\
        0 & 0 & 1 & -1 &\cdots & \cdots & 0\\
        \vdots &&&&\ddots&& \vdots\\
        0 & 0 & 0 & 0 &\cdots & (-1)^{n-1} & (-1)^n\\
        \end{array} } \right]
    \]
    Using this matrix and a shift vector, we define the following linear transformation on $\vec{x}$.
    \[
        \left[ {\begin{array}{c}
            v_1\\
            v_2\\
            v_3\\
            v_4\\
            \vdots\\
            v_{n-1}\\
            \end{array} } \right]=
            \boldsymbol{A}
        \left[ {\begin{array}{c}
        x_2\\
        x_3\\
        x_4\\
        x_5\\
        \vdots\\
        x_n\\
        \end{array} } \right]+
        \left[ {\begin{array}{c}
        -x_1\\
        1\\
        0\\
        1\\
        0\\
        \vdots\\
        \end{array} } \right]
    \]
    
    Since the scaling matrix $\boldsymbol{A}$ is lower triangular, its determinant is the product of the entries in its diagonal. Therefore $|\det(\boldsymbol{A})|=1$. This means that the volume of the polyhedron defined by (\ref{thm:ladderprob}) is preserved. The vector $\vec{v}$ generated by this transformation on the truncated version of $\vec{x}$ is shown below.
    \[
        \left[ {\begin{array}{c}
            v_1\\
            v_2\\
            v_3\\
            v_4\\
            v_5\\
            \vdots\\
            \end{array} } \right]=
        \left[ {\begin{array}{c}
        x_2-x_1\\
        1-x_3+x_2\\
        x_4-x_3\\
        1-x_5+x_4\\
        x_6-x_5\\
        \vdots\\
        \end{array} } \right]
    \]
    Applying this transformation to the inner integrals of (\ref{thm:ladderprob}) yeilds the following.
    \[
            \int_{x_1}^{x_1+1}
                \int_{x_1+1}^{x_2+1}
                    \int_{x_2+1}^{x_3+1}
                        \cdots
                \,dx_4
            \,dx_3
        \,dx_2
        =
            \int_{0}^{1}
                \int_{v_1}^{0}
                    \int_{v_2}^{1}
                        \cdots
                \,dv_3
            \,dv_2
        \,dv_1\]
        Depending on the parity of $n$, this integral may be positive or negative, but since the bounds are all the same, just in a different order, the integrals are equal up to a factor of $-1$.
        \[
        \bigg|\int_{0}^{1}
        \int_{v_1}^{0}
            \int_{v_2}^{1}
                \cdots
        \,dv_3
    \,dv_2
\,dv_1\bigg|=\int_{0}^{1}
        \int_{0}^{v_1}
            \int_{v_2}^{1}
                \cdots
            \,dv_3
        \,dv_2
    \,dv_1=\frac{A_{n-1}}{(n-1)!}\]
    We replace the inner integrals from (\ref{thm:ladderprob}) with the previous result.
    \[\frac{L^n}{n!} \Pr(P)=
    \int_{0}^{L}
        \frac{A_{n-1}}{(n-1)!}
    \,dx_1+e
    =
    \frac{A_{n-1}}{(n-1)!}L+e\]
    Therefore we have shown that $\Pr(P)=\frac{n A_{n-1}L+e}{L^n}$ which implies the conclusion of the theorem.

    \[L^n \Pr(P)\sim n A_{n-1}L\]
\end{proof}

\section{Open problems}

We have yet to show non-integral formulas for any bricks which are not ladders. Though, for some ladders with twins, our data suggests their exponential generating functions. 

Let a semiorder on $n$ elements be a \emph{ladder with a twin in the $i^{th}$ position} if when one of the twins is removed, the semiorder is the ladder on $n-1$ elements and the remaining twin is in the $i^{th}$ position in every endpoint linear extension. From the probabilities we have calculated, the pattern suggests the following.
\begin{enumerate}
    \item If $P$ is a ladder with a twin in the $1^{st}$ position, then $L^n\Pr(P)\sim nA_{n-1}L$ where $A_n$ has exponential generating function $(x-1)(\sec(x)+\tan(x))$.
    \item If $P$ is a ladder with a twin in the $2^{nd}$ position, then $L^n\Pr(P)\sim nA_{n-1}L$ where $A_n$ has exponential generating function $(2-x)(\sec(x)+\tan(x))$.
\end{enumerate}
These patterns hold through at least $n=20$ for both of these cases. Unfortunately, an exponential generating function for ladders with twins in the $3^{rd}$ position or higher have not been conjectured. The natural variations on the exponential generating functions listed above do not satisfy our calculated probabilities.


\bibliographystyle{plain}
\bibliography{references}

\end{document}